\documentclass[12pt]{amsart}
\usepackage{amsmath,latexsym,amscd,amsbsy,amssymb,amsfonts,amsthm,fleqn,leqno,
euscript, graphicx, texdraw}
\numberwithin{equation}{section}

\newcommand{\I}{\mathbb I}

\newcommand{\mesh}{\mathrm{mesh}}

\newtheorem{thm}{Theorem}[section]
\newtheorem{pro}[thm]{Proposition}

\newtheorem{cor}[thm]{Corollary}

\newtheorem{qu}[thm]{Question}

\theoremstyle{definition}
\newtheorem{dfn}[thm]{Definition}

\theoremstyle{remark}

\begin{document}

\title[Generalized Cantor manifolds and indecomposable continua]
{Generalized Cantor manifolds and indecomposable continua}

\author{V. Todorov}
\address{Department of Mathematics, UACG, $1$ H. Smirnenski blvd.,
$1046$ Sofia, Bulgaria} \email{vtt-fte@uacg.bg}

\author{V. Valov}
\address{Department of Computer Science and Mathematics,
Nipissing University, 100 College Drive, P.O. Box 5002, North Bay,
ON, P1B 8L7, Canada} \email{veskov@nipissingu.ca}

\date{\today}
\thanks{The second author was partially supported by NSERC
Grant 261914-08.}


\begin{abstract}
We review results concerning homogeneous compacta and discuss some open questions.
It is established that indecomposable continua are Alexandroff (resp., Mazurkiewicz, or strong Cantor) manifolds
with respect to the class of all continua. We also provide some new proofs of Bing's theorems about separating
metric compacta by hereditarily indecomposable compacta.
\end{abstract}
\maketitle\markboth{}{Cantor manifolds and indecomposable continua}





\section{Introduction}

Cantor manifolds were introduced by Urysohn \cite{u} as a
generalization of Euclidean manifolds. Recall that a space $X$ is a
{\em Cantor $n$-manifold} if $X$ cannot be separated by a closed
$(n-2)$-dimensional subset. In other words, $X$ cannot be the union
of two proper closed sets whose intersection is of covering
dimension $\leq n-2$. Another specification of Cantor manifolds was
considered by Had\v{z}iivanov \cite{h}: $X$ is a {\em strong Cantor
$n$-manifold} if for arbitrary representation
$X=\bigcup_{i=1}^{\infty}F_i$, where all $F_i$ are proper closed
subsets of $X$, we have $\dim(F_i\cap F_j)\geq n-1$ for some $i\neq
j$.
Had\v{z}iivanov and Todorov \cite{ht} introduced the class of Mazurkiewicz $n$-manifolds, which
is a proper sub-class of the strong Cantor
$n$-manifolds: $X$ is a {\em Mazurkiewicz $n$-manifold} if for any disjoint closed massive subsets
$A$ and $B$ of $X$ (a massive subset of $X$ is a set with non-empty interior in $X$), and any normally placed set $M\subset X$ with
$\dim M\leq n-2$, there exists a continuum in $X\backslash M$ intersecting $A$ and $B$ (equivalently, no such
$M$ is cutting $X$ between $A$ and $B$). Here, $M$ is normally placed in $X$ provided every two
disjoint closed in $M$ sets have disjoint open in $X$ neighborhoods. For example, every $F_\sigma$-subset of a normal space is
normally placed in that space.
The notion of a Mazurkiewicz $n$-manifold has its roots in  the
classical Mazurkiewicz theorem saying that no region in the
Euclidean $n$-space can be cut by a  subset of dimension $\leq
n-2$~\cite{ma} (recall that a set $M\subset X$ does not cut $X$ if for every two points $x,y\in X\backslash M$
there is a continuum $K\subset X\backslash M$ joining $x$ and $y$).

Alexandroff \cite{ps} introduced the stronger
notion of $V^n$-continua: a compactum $X$ is a {\em $V^n$-continuum}
if for every two closed disjoint subsets $X_0$, $X_1$ of $X$, both
having non-empty interior in $X$, there exists an open cover
$\omega$ of $X$ such that there is no partition $P$ in $X$ between
$X_0$ and $X_1$ admitting an $\omega$-map into a space $Y$ with
$\dim Y\leq n-2$ ($f\colon P\to Y$ is said to be an $\omega$-map if there exists
an open cover $\gamma$ of $Y$ such that $f^{-1}(\gamma)$ refines $\omega$).

Obviously, strong Cantor $n$-manifolds are Cantor $n$-manifolds.
Moreover,  every $V^n$-continuum is a Mazurkiewicz $n$-manifold and Mazurkiewicz $n$-manifolds
are strong Cantor $n$-manifolds, see
\cite{ht}. None of the above inclusions is reversible, see
\cite{kktv}.

In the present paper we consider
quite general concepts of the above notions with respect to some classes of
finite or infinite-dimensional  spaces.

Let $\mathcal{C}$ be a class of topological spaces.

\begin{dfn}\label{dfn1}
A space $X$ is an {\em Alexandroff manifold with respect to $\mathcal C$}
(br., {\em Alexandroff $\mathcal C$-manifold}) if for every two closed, disjoint, massive subsets $X_0$, $X_1$ of $X$
there exists an open cover
$\omega$ of $X$ such that there is no partition $P$ in $X$ between
$X_0$ and $X_1$ admitting an $\omega$-map onto a space $Y\in\mathcal C$.
\end{dfn}

\begin{dfn}\label{dfn2}
A space $X$  is said to be a {\em Mazurkiewicz manifold with respect to
$\mathcal{C}$} (br., {\em Mazurkiewicz $\mathcal C$-manifold}) provided for every two closed, disjoint,
massive subsets $X_0,X_1\subset X$,
and every set $F=\bigcup_{i=0}^\infty F_i\subset X$ with each $F_i\in\mathcal{C}$ being proper closed subset of $X$,
there exists a continuum $K$ in $X\setminus F$ joining $X_0$ and $X_1$.

If, under the above conditions, for every open cover $\omega$ of $X$ there exists a set $K_\omega\subset X\backslash F$
joining $X_0$ and $X_1$ such that $K_\omega$ admits an $\omega$-map onto a continuum, $X$ is called a {\em weak Mazurkiewicz $\mathcal C$-manifold}.
\end{dfn}

\begin{dfn}
A space $X$  is a {\em strong Cantor manifold with respect to
$\mathcal{C}$} (br., {\em strong Cantor $\mathcal C$-manifold}) if $X$ can not be represented as
the union
\begin{align}\displaystyle \label{eq1}& X=\bigcup_{i=0}^\infty F_i \quad \text{with}\quad \bigcup_{i\ne j}(F_i\cap F_j)\in \mathcal C,\\
&\text{where all $F_i$ are proper closed subsets of  $X$}.
\notag\end{align}
 \end{dfn}

\begin{dfn}
A space $X$  is a {\em  Cantor manifold with respect to  $\mathcal{C}$} (br., {\em Cantor $\mathcal C$-manifold}) if $X$ cannot be separated by a
closed subset which belongs to $\mathcal{C}$.
\end{dfn}

Obviously, any strong Cantor $\mathcal C$-manifold is a Cantor $\mathcal C$-manifold. Moreover, if the class $\mathcal C$ is hereditary with respect to $F_\sigma$-subsets, then  compact Mazurkiewicz $\mathcal C$-manifolds are strong Cantor $\mathcal C$-manifolds, see \cite{kktv}.

\section{Generalized Cantor manifolds and homogeneous continua}

In this section we discuss some properties of homogeneous continua and ask some questions.

We first introduce the
general dimension function $D_{\mathcal{K}}$ considered in \cite{kktv}, which captures the
covering dimension, cohomological dimension $\dim_G$ with respect to
any Abelian group $G$, as well as the extraordinary dimension
$\dim_L$ with respect to a given $CW$-complex $L$.

A sequence $\mathcal{K}=\{K_0,K_1,..\}$ of
$CW$-complexes is called a {\em stratum} for a dimension theory
\cite{dr} if
\begin{itemize}\item
for each space $X$ admitting a perfect map onto a metrizable space,
$K_n\in AE(X)$ implies both $K_{n+1}\in AE(X\times\I)$ and
$K_{n+j}\in AE(X)$ for all $j\geq 0$.
\end{itemize}
Here, $K_n\in AE(X)$ means that $K_n$ is an absolute extensor for
$X$. Given a stratum $\mathcal{K}$, we can define a dimension
function $D_{\mathcal{K}}$ in a standard way:
\begin{enumerate}
\item
$D_{\mathcal{K}}(X)=-1$ iff $X=\emptyset$;
\item $D_{\mathcal{K}}(X)\le n$ if
$K_n\in AE(X)$ for $n\ge 0$; if $D_{\mathcal{K}}(X)\le n$ and
$K_m\not\in AE(X)$ for all $m<n$, then $D_{\mathcal{K}}(X)= n$;
\item
$D_{\mathcal{K}}(X)=\infty$ if $D_{\mathcal{K}}(X)\le n$ is not
satisfied for any $n$.
\end{enumerate}

 If $\mathcal{K}=\{\mathbb{S}^0,\mathbb{S}^1,..\}$, we
obtain the covering dimension $\dim$. The stratum
$\mathcal{K}=\{\mathbb{S}^0,K(G,1),..,K(G,n),..\}$,  $K(G,n)$,
$n\geq 1$, being the Eilenberg-MacLane complexes for a given group
$G$, determines the cohomological dimension $\dim_G$. Moreover, if
$L$ is a fixed $CW$-complex and
$\mathcal{K}=\{L,\Sigma(L),..,\Sigma^n(L),..\}$, where $\Sigma^n(L)$
denotes the $n$-th iterated suspension of $L$, we obtain the
extraordinary dimension $\dim_L$ introduced  by Shchepin
\cite{es:98} and considered in details by Chigogidze \cite{ch:03}.

According to the countable sum theorem in extension theory, it
follows directly from the above definition that
$D_{\mathcal{K}}(X)\leq n$ implies $D_{\mathcal{K}}(A)\leq n$ for
any $F_{\sigma}$-subset $A\subset X$.

Henceforth, $\mathcal{C}$ will denote one of the four classes:
\begin{itemize}
\item
 the class $\mathcal D_{\mathcal K}^k$ of at most
$k$-dimensional spaces with respect to dimension $D_{\mathcal{K}}$,

\item the class $\mathcal D_{\mathcal{K}}^{<\infty}$ of strongly countable
$D_{\mathcal K}$-dimensional spaces, i.e. all spaces represented as
a countable union of closed finite-dimensional subsets with respect
to $D_{\mathcal{K}}$,

\item
the class $\mathbf C$ of paracompact $C$-spaces,

 and
\item the class $\mathcal {WID}$
of weakly infinite-dimensional spaces.
\end{itemize}

For definitions of a weakly (strongly) infinite-dimensional or a $C$-space, see \cite{re:95}.

It was proved in~\cite{kru} that every homogeneous metrizable,
locally compact, connected space $X$ with the covering dimension
$\dim X=n\le\infty$ is a Cantor $n$-manifold; in case where $X$ is
strongly  infinite-dimensional, it is a
Cantor $\mathcal {WID}$-manifold. Next theorem, established in \cite{kktv},
significantly generalizes  those  results.

\begin{thm}
Every metrizable homogeneous continuum $X\notin\mathcal{C}$ is a
strong Cantor $\mathcal{C}$-manifold  provided that:
 \begin{enumerate}
\item
  $\mathcal{C}$ is any
of the  following three classes:  $\mathcal{WID}$, $\mathbf C$,
$\mathcal D_{\mathcal K}^{n-2}$
 $($in the latter case  we additionally assume
$D_{\mathcal{K}}(X)=n$$)$;\\
or
 \item
$\mathcal C=\mathcal D_{\mathcal K}^{<\infty}$ and   $X$ does not
contain closed subsets of arbitrary large finite dimension
$D_{\mathcal K}$.
\end{enumerate}
 \end{thm}

In case $X$ is locally connected, Theorem 2.1 was generalized in \cite{kv}.

\begin{thm}
Let $X$ be a homogeneous locally compact, locally connected metric space. Suppose
$U$ is a region in $X$ with $U\not\in\mathcal C$, where
$\mathcal C\in\{\mathcal {WID}, \mathbf{C}, \mathcal D_{\mathcal{K}}^{n-2}, D_{\mathcal{K}}^{<\infty}\}$,
$n\geq 1$, and $D_{\mathcal{K}}(U)=n$ in the case $\mathcal C=D_{\mathcal{K}}^{n-2}$. Then
$U$ can not be cut by any set $\bigcup_{i=0}^\infty F_i$ with each $F_i\in\mathcal C$ being closed in $U$.
\end{thm}

Theorems 2.1 and 2.2 are based on the following results from \cite{kktv}:

\begin{thm}
Let $X$ be a compact space.
\begin{enumerate}
\item If $D_{\mathcal{K}}(X)=n$, then $X$ contains a closed
subset $M$ such that $D_{\mathcal{K}}(M)=n$ and $M$ is both Alexandroff
and a Mazurkiewicz manifold with respect to the
class $\mathcal D_{\mathcal{K}}^{n-2}$;
\item If $D_{\mathcal{K}}(X)=\infty$,
then either $X$ contains closed subsets of arbitrary large finite
dimension $D_{\mathcal{K}}$ or $X$ contains a compact Mazurkiewicz
$\mathcal D_{\mathcal{K}}^{<\infty}$-manifold;
\item If $X\not\in\mathbf C$, then it contains a compact Mazurkiewicz $\mathbf C$-manifold;
\item If $X$ is metrizable and strongly infinite-dimensional, then it contains a
compact Mazurkiewicz $\mathcal {WID}$-manifold.
\end{enumerate}
\end{thm}

Some particular cases of Theorem 2.2 were established by different authors, see
\cite{ps1}, \cite{hs}, \cite{hm}, \cite{ku}, \cite{tu}, \cite{tu1}.

Here is the main questions arising from the above results.

\begin{qu}
Let $X$ be a homogeneous compact metric space and $\mathcal C\in\{\mathcal {WID}, \mathbf{C}, \mathcal D_{\mathcal{K}}^{n-2}, D_{\mathcal{K}}^{<\infty}\}$, where
$n\geq 1$ and $D_{\mathcal{K}}(U)=n$ in the case $\mathcal C=D_{\mathcal{K}}^{n-2}$. Is $X$  an Alexandroff $\mathcal C$-manifold?
What is the answer of the above question if, in addition, $X$ is locally connected?
\end{qu}

Krupski \cite{kr1} conjectured that any $n$-dimensional, homogeneous metric $ANR$-continuum is a $V^n$-continuum.
Next result, which is still unpublished, provides a partial solution of Krupski's conjecture.

\begin{thm}
Let $X$ be a homogeneous, metric $n$-dimensional $ANR$-continuum with $H^n(X,\mathbb Z)\neq 0$. Then $X$ is a $V^n$-continuum.
\end{thm}

As it was mentioned above, if $\mathcal C$ is the class of all spaces of covering dimension $\leq n-2$, then every Alexandroff $\mathcal C$-manifold is a Mazurkiewicz $\mathcal {C}$-manifold. So, we have next question.
\begin{qu}
Let $\mathcal C$ be one of the above four classes.
Is there any Alexandroff $\mathcal C$-manifold which is not a Mazurkiewicz $\mathcal {C}$-manifold?
\end{qu}

\begin{qu}
Let $X$ and $\mathcal C$ be as in Question $2.4$. Is  $X$ a Mazurkiewicz $\mathcal {C}$-manifold?
More generally, is it true that $X$ satisfy the conclusion in Theorem $2.2$ with $U=X$?
\end{qu}

Having in mind Theorem 2.3, next question is also interesting.

\begin{qu}
Let $\mathcal C\in\{\mathbf C, \mathcal{WID}\}$ and $X$ be a compact metric space with $X\not\in\mathcal C$. Does $X$ contain
a compact Alexandroff $\mathcal C$-manifold?
\end{qu}

The last question in this section is inspired by the following result of the first author \cite{vt}: If $M\subset \I^n$ is a set of dimension
$\dim M\leq k$, where $k\leq n-2$, then every two different points $x,y\in\I^n\backslash M$ can be joined by a $V^{n-k-1}$-continuum
$K\subset\I^n\backslash M$.

\begin{qu}
Let $X$ be a homogeneous metric (locally connected) continuum and $M\subset X$ an $F_\sigma$-set with $\dim F\leq k\leq n-2$. Is it true that any
two massive closed disjoint sets $A, B\subset X$ can be joined in $X\backslash M$ by a  $V^{n-k-1}$-continuum?
\end{qu}

We conclude this section with a few words about our definition  of Alexandroff $\mathcal C$-manifolds. The Alexandroff
definition of $V^n$-continua is based on the following property of the covering dimension: if $X$ is a paracompact space and for
every open cover $\omega$ of $X$ there exists an $\omega$-map $X$ onto a paracompact space $Y$ with $\dim Y\leq n$, then $\dim X\leq n$.
It follows from \cite[Theorem 2.4]{cv} that the dimension $D_{\mathcal{K}}$ has a similar property for spaces $X$ admitting perfect maps onto
metrizable spaces. Unfortunately, this is not true for the classes $\mathbf C$ and $\mathcal{WID}$. For example, for any open cover
 $\omega$ of the Hilbert cube $Q$ there exists an $\omega$-map onto a finite-dimensional space, but $Q$ is strongly infinite-dimensional.

\section{Indecomposable continua and Cantor manifolds}

Recall that a continuum is indecomposable if it is not the union of two proper sub-continua.
The results in this section came out from the observation \cite{Ku1} that any indecomposable continuum can not be separated by a
proper connected subset. This means that indecomposable continua are Cantor manifolds with respect to the class $\mathfrak{K}$ of all continua.

\begin{thm}
Any metric indecomposable continuum is both Mazurkiewicz $\mathfrak{K}(\aleph_0)$-manifold and a strong Cantor $\mathfrak{K}(\aleph_0)$-manifold, where
$\mathfrak{K}(\aleph_0)$ is the class of all spaces having at most countably many components.
\end{thm}

\begin{proof}
Suppose $X$ is a metric indecomposable continuum and let $A$ and $B$ be  disjoint, closed, massive subsets of $X$.
Since there are uncountably many disjoint composants of $X$, see \cite{Ku1}, for any sequence $\{F_i\}_{i\geq 1}$ of proper closed subsets of $X$ with $F_i\in\mathfrak{K}(\aleph_0)$, $i\geq 1$, there exists a composant $K$ of $X$ disjoint from the set $F=\bigcup_{i=1}^{\infty}F_i$. Because each composant is dense in $X$ \cite{Ku1}, there exist points $a\in A\cap K$ and $b\in B\cap K$. Finally, using that any composant of a metric continuum is a countable union of proper sub-continua, we find a continuum $P\subset X\backslash F$ containing $a, b$. Therefore, $X$ is a Mazurkiewicz $\mathfrak{K}(\aleph_0)$-manifold.

Assume that $X$ is not a  strong Cantor $\mathfrak{K}(\aleph_0)$-manifold. So, $X=\bigcup_{i=1}^{\infty}H_i$ for some sequence $\{H_i\}_{i\geq 1}$ of proper closed subset of $X$ such that $H_i\cap H_j\in\mathfrak{K}(\aleph_0)$ for all $i\neq j$. According to the Baire theorem, there are two different integers $n,m$ such that both $H_m$ and $H_n$ have non-empty interior. Since $X$ is a Mazurkiewicz $\mathfrak{K}(\aleph_0)$-manifold, there exists a continuum
$C\subset X\backslash\bigcup_{i\neq j}H_i\cap H_j$ with $H_n\cap C\neq\varnothing\neq H_m\cap C$. This implies that $C$ is a non-degenerate continuum covered by the disjoint family $\{C\cap H_i:i\geq 1\}$ of closed sets at least two of which are non-empty. This contradicts the Sierpi\'{n}ski theorem. Hence, $X$ is a  strong Cantor $\mathfrak{K}(\aleph_0)$-manifold.
\end{proof}

\begin{cor}
Any hereditarily indecomposable continuum is a  weak Mazurkiewicz $\mathfrak{K}(\aleph_0)$-manifold. \end{cor}

\begin{proof}
According to Theorem 3.1, we may assume that $X$ is non-metrizable. Then, by \cite[Proposition 4.2]{hp}, $X$ is the limit space of an inverse system $S=\{X_\sigma,\pi^\sigma_\rho, \Sigma\}$ consisting of metric hereditarily indecomposable continua $X_\sigma$. Since the limit space of each inverse sequence of hereditarily indecomposable continua is also hereditarily indecomposable, we may assume that $S$ is $\sigma$-continuous (i.e., $\beta=\sup\{\alpha_n:n\geq 1\}\in\Sigma$ and
$X_\beta=\displaystyle\lim_\leftarrow\{X_{\alpha_n},\pi^{\alpha_{n+1}}_{\alpha_n}\}$ for
every countable chain $\{\alpha_n:n\geq 1\}$ in $\Sigma$).
Denote by $\pi_\sigma\colon X\to X_\sigma$, $\sigma\in\Sigma$, the limit projections.
Let $A$ and $B$ be  disjoint, closed and massive subsets of $X$, $\{F_i\}_{i\geq 1}$ a sequence of proper closed subsets of $X$ with $F_i\in\mathfrak{K}(\aleph_0)$ for all $i$, and $\omega$ be an open cover of $X$. There exists $\alpha\in\Sigma$, an open cover $\omega_\alpha$ of $X_\alpha$ and disjoint open sets $U_A$ and $U_B$ in $X_\alpha$ such that $\pi_\alpha^{-1}(U_A)\subset A$, $\pi_\alpha^{-1}(U_B)\subset B$,
$\pi_\alpha^{-1}(\omega_\alpha)$ refines $\omega$ and $\pi_\alpha(F_i)\neq X_\alpha$ for all $i$. Such $\alpha$ exists because $S$ is $\sigma$-continuous. Since $X_\alpha$ is a Mazurkiewicz $\mathfrak{K}(\aleph_0)$-manifold, there is a continuum $P\subset X_\alpha\backslash\bigcup_{i=1}^{\infty}\pi_\alpha(F_i)$ joining $U_A$ and $U_B$. Then the set $K=\pi_\alpha^{-1}(P)\subset X\backslash\bigcup_{i=1}^{\infty}F_i$ is joining $A$ and $B$, and $\pi_\alpha\colon K\to P$ is an $\omega$-map. This completes the proof.
\end{proof}

\begin{thm}
Any metric indecomposable continuum is an Alexandroff $\mathfrak{K}$-manifold.
\end{thm}

\begin{proof}
Suppose there exists an indecomposable metric continuum $(X,d)$ which is not an Alexandroff $\mathfrak{K}$-manifold. Consequently, we have two closed,
disjoint, massive sets $A$ and $B$ in $X$ satisfying the following condition: for every open cover $\omega$ of $X$ there exists a partition $P_\omega$ in $X$
between $A$ and $B$, and a surjective $\omega$-map $g_\omega\colon P_\omega\to Y_\omega$, where $Y_\omega$ is a continuum.
Since $X$ is a metric compactum, there is a sequence $\{\omega_n\}$ of open covers of $X$ such that $\mesh(\omega_n)<1/n$ and the sequence $\{P_{\omega_n}\}$ converges to the compact set $P_0\subset X$ with respect to the Hausdorff metric generated by $d$. Let $a\in\mathrm{Int}(A)$
and $b\in\mathrm{Int}(B)$. Because each $P_{\omega_n}$ is disjoint from $A\cup B$, $a, b\not\in P_0$.

\textit{Claim $1$. $P_0$ is connected}.

Indeed, otherwise let $P_0=P_1\cup P_2$ be the union of two disjoint non-empty closed subsets. Choose two open sets $U_1, U_2$ in $X$ having disjoint
closures such that
$P_j\subset U_j$, $j=1,2$. Then there exists $m$ such that $P_{\omega_m}\subset U_1\cup U_2$ and $P_{\omega_m}\cap U_j\neq\varnothing$ for each $j=1,2$. We may assume that $m$ is so big that the fibers of the map $g_{\omega_n}$ meets only one of the sets $U_1, U_2$. The last condition implies
that $g_{\omega_m}(P_{\omega_m}\cap\overline{U_1})$ and $g_{\omega_m}(P_{\omega_m}\cap\overline{U_2})$ are non-empty disjoint subsets of
$Y_{\omega_m}$ whose union is $Y_{\omega_m}$, a contradiction.

\textit{Claim $2$. $P_0$ is a partition of $X$}.

Assume $P_0$ is not a partition of $X$, and choose two closed  sets $A_1, B_1$ in $X$ both disjoint from $P_0$ and having non-empty interior such that $a\in A_1\subset A$ and $b\in B_1\subset B$. Since $X$ is a Mazurkiewicz $\mathfrak{K}(\aleph_0)$-manifold, there exists a continuum $K\subset X\backslash P_0$ joining $A_1$ and $B_1$. So, $P_0\subset X\backslash K$, which yields that $P_{\omega_s}\subset X\backslash K$ for some $s$. The last inclusion contradicts the fact that $P_{\omega_s}$ is a partition of $X$ between $A$ and $B$.

Hence, $P_0$ is a proper sub-continuum of $X$ separating $X$, which is impossible because $X$ is indecomposable.
\end{proof}

\begin{qu}
Let $X$ be an indecomposable continuum. Is it true that:
\begin{itemize}
\item $X$ is a strong Cantor $\mathfrak{K}(\aleph_0)$-manifold;
\item $X$ is an Alexandroff $\mathfrak{K}$-manifold?
\end{itemize}
\end{qu}

Next two propositions provide more examples of hereditarily indecomposable continua, and generalize the well known facts \cite{b}
that every $(n+1)$-dimensional (resp., strongly infinite-dimensional) metric compactum contains an $n$-dimensional
(resp., strongly infinite-dimensional) hereditarily indecomposable continuum (see also \cite{hmp} for another proofs).

\begin{pro}
Let $X$ be a metric compactum of dimension $D_{\mathcal{K}}(X)=n+1$, where $n\geq 0$ and $\mathcal{K}=\{K_0,K_1,..\}$ is a given stratum.
Then $X$ contains a hereditarily indecomposable continuum $X_0$ of dimension $D_{\mathcal{K}}(X_0)\in\{n,n+1\}$.
\end{pro}

\begin{proof}
By \cite{le}, there exists a map $g\colon X\to\mathbb I$ such that all components of the fibers $g^{-1}(t)$, $t\in\mathbb I$, are
hereditarily indecomposable. Let $g=h\circ p$ be the monotone-light decomposition of $g$ with $p\colon X\to Y$ being monotone and
$h\colon Y\to\mathbb I$ light. Then, by Hurewicz's theorem, $\dim Y\leq 1$. Assuming that $D_{\mathcal{K}}(p^{-1}(y))\leq n-1$
for all $y\in Y$, according to \cite[Theorem 2.4]{cv1}, we obtain $D_{\mathcal{K}}(X)\leq n$. So, there exists $y_0\in Y$ with
$D_{\mathcal{K}}(p^{-1}(y_0))\geq n$. Then $X_0=p^{-1}(y_0)$ is the required continuum.
\end{proof}

Similarly, one can prove the following propositions.

\begin{pro}
Let $X$ be a strongly infinite-dimensional $($resp., weakly infinite-dimensional with $X\not\in\bf{C}$$)$
metric compactum.
Then $X$ contains a strongly infinite-dimensional hereditarily indecomposable continuum $X_0$ $($resp., weakly infinite-dimensional
hereditarily indecomposable continuum $X_0\not\in\bf{C}$$)$.
\end{pro}

Dranishnikov's example \cite{d} of a strongly infinite-dimensional metric compactum having cohomological dimension $\dim_{\mathbb Z}=3$
implies next corollary.

\begin{cor}
There exists a strongly infinite-dimensional hereditarily indecomposable continuum $X$ with $\dim_{\mathbb Z}(X)\in\{2,3\}$.
\end{cor}

Next proposition is analogue of Bing's partition theorems \cite{b} (recall that a compactum $X$ is called hereditarily indecomposable provided
each continuum in $X$ is indecomposable).

\begin{pro}
Let $f\colon X\to Y$ be a perfect map between metric spaces with $X$ being connected, and let $A, B$ be two closed disjoint subsets of $X$. Then there is a closed partition $H$ of $X$ between $A$ and $B$ with the following properties:
\begin{itemize}
\item[(i)] The intersection $f^{-1}(y)\cap H$ is hereditarily indecomposable for every $y\in Y$;
\item[(ii)] If $K$ is a continuum contained in some $f^{-1}(y)$ such that $K\cap (A\cup B)\neq\varnothing$, then $K$ contains a component
of the set $f^{-1}(y)\cap H$.
\end{itemize}
\end{pro}

\begin{proof}
Let $h\colon X\to\mathbb I$ be a continuous function with $h(A)=0$ and $h(B)=1$. According to \cite{vv}, there exists a function
$g\colon X\to\mathbb I$ such that $|g(x)-h(x)|<1/4$ for all $x\in X$ and the restrictions $g_y=g|f^{-1}(y):f^{-1}(y)\to\I$, $y\in Y$,
satisfy the following conditions: the fibers of $g_y$ are hereditarily
indecomposable, and any continuum $K\subset f^{-1}(y)$ either is contained in a fiber of $g_y$ or contains a component of a fiber of $g_y$. Then $U_A=g^{-1}([0,1/2))$ and $U_B=g^{-1}((1/2, 1])$ are disjoint neighborhoods of $A$ and $B$, respectively. Moreover,
$H=X\backslash (U_A\cup U_B)=g^{-1}(1/2)$, which implies that $H$ has the desired properties.
\end{proof}

Below $\mathfrak{In}$ denotes either the class of indecomposable continua or the class of hereditarily indecomposable compacta.
\begin{pro}
Let $X\in\mathfrak{In}$  and $f\colon X\to Y$ be a surjective map. Then there exists a compactum $Z\in\mathfrak{In}$ and
maps $g\colon X\to Z$, $h\colon Z\to Y$ such that $w(Z)=w(Y)$. Moreover, we can have $D_{\mathcal{K}}(Z)\leq D_{\mathcal{K}}(X)$.
\end{pro}

\begin{proof}
K.P. Hart and E. Pol \cite[Proposition 4.2]{hp}, using
elementary substructures and L\"{o}wenheim-Skolem theorem, proved a similar factorization theorem for hereditarily indecomposable continua and the covering dimension $\dim$. The first part of the proposition, concerning indecomposable continua, can be obtained applying Hart-Pol's arguments and the following characterization of indecomposable continua \cite[Theorem 1.3]{wl}: A continuum $X$ is indecomposable if and only if for any nonempty open sets $U$ and $V$ in $X$, there exist closed sets $A$ and $B$ such that $X=A\cup B$, $A\cap B\subset U$, $A\cap V\neq\varnothing\neq B\cap V$. Using this fact and  existence of a factorization theorem for extension dimension (see \cite{lrs}, or \cite{pa}), we construct an inverse sequence $S=\{Z_n, p^m_n\}$ of compacta and  maps $g_n\colon X\to Z_n$, $h_n\colon Z_n\to Y$ satisfying the following conditions for all $k,n\in\mathbb N$:
\begin{itemize}
\item $w(Z_n)=w(Y)$;
\item $Z_{2k+1}$ is indecomposable (resp., hereditarily indecomposable);
\item $D_{\mathcal{K}}(Z_{2k})\leq D_{\mathcal{K}}(X)$;
\item $h_1\circ g_1=f$, $g_{n}=p^{n+1}_n\circ g_{n+1}$ and $h_{n+1}=h_n\circ p^{n+1}_n$.
\end{itemize}
Then $Z=\displaystyle\lim_{\leftarrow} S$ is a compactum of weight $w(Z)=w(Y)$ and $D_{\mathcal{K}}(Z)\leq D_{\mathcal{K}}(X)$. Moreover, the maps $g_n$ provide
a map $g\colon X\to Z$ such that $f=h\circ g$, where $h=h_1\circ p_1$ with $p_1$ being the projection from $Z$ onto $Z_1$. Finally, since $Z$ is the limit of the inverse system $\{Z_{2k+1}, p^{2m+1}_{2k+1}\}$ consisting of compacta from $\mathfrak{In}$, $Z$ is also from $\mathfrak{In}$.
\end{proof}

\begin{cor}
A compactum $X$ belongs to the class $\mathfrak{In}$ if and only if $X$ is the limit space of an inverse system
of metric compacta from $\mathfrak{In}$.
\end{cor}


\end{document}